\documentclass[11pt]{amsart}
\pdfoutput=1
\usepackage{amsfonts}
\usepackage{amsmath}
\usepackage{mathtools}
\usepackage{amsthm}
\usepackage{amssymb}
\usepackage{mathrsfs}
\usepackage[fit]{truncate}
\usepackage[margin=1.1in]{geometry}
\usepackage{varioref}
\makeatletter
\labelformat{equation}{\tagform@{#1}}
\makeatother
\usepackage{bm}
\usepackage{tikz-cd}
\usepackage[makeroom]{cancel}
\usepackage{enumitem} 
\usepackage{multirow}
\newcommand{\ul}[1]{\underline{\vphantom{\mathstrut}#1}}
\newcommand{\ol}[1]{\overline{\vphantom{\mathstrut}#1}}
\newcommand{\bmm}[1]{\bm{\mathcal{#1}}}
\newcommand{\rr}{\mathbb{R}^n \oplus (\mathbb{R}^n)^*}
\renewcommand{\gg}{\mathfrak{g} \oplus \mathfrak{g}^*}
\makeatletter
\renewcommand{\@biblabel}[1]{[#1]\hfill}
\makeatother
\numberwithin{equation}{section}
\newcommand{\upperset}[2]{%
    \underset{%
        \text{\raisebox{1ex}{\smash{\fontsize{5}{5}$#1$}}}
              }{#2}%
                         }

\newcommand{\R}{\mathbb{R}}
\newcommand{\N}{\mathbb{N}}
\newcommand{\ric}{\operatorname{Ric}}
\newcommand{\rc}{\operatorname{Rc}}
\newcommand{\D}{\mathcal{D}}
\newcommand{\Mu}{\bm{\mu}}
\newcommand{\mL}{\mathcal{L}}
\newcommand{\G}{\mathcal{G}}

\theoremstyle{plain}
\newtheorem{theorem}{Theorem}[section]

\newtheorem{lemma}[theorem]{Lemma}
\newtheorem{proposition}[theorem]{Proposition}
\theoremstyle{definition}
\newtheorem{definition}[theorem]{Definition}
\newtheorem{example}[theorem]{Example}
\newtheorem{remark}[theorem]{Remark}

\begin{document}

\title{Generalized Ricci flow on nilpotent Lie groups}
\author{Fabio Paradiso}
\address{Dipartimento di Matematica ``G.~Peano'' \\ Universit\`a degli Studi di Torino\\
Via Carlo Alberto 10\\
10123 Torino\\ Italy}
\email{fabio.paradiso@unito.it}
\date{}
\subjclass[2010]{53D18, 53C44, 53C30}
\keywords{Generalized geometry, Generalized Ricci flow, Nilpotent Lie groups}

\begin{abstract}
We define solitons for the generalized Ricci flow on an exact Courant algebroid, building on the definitions of \cite{Str17,Gar19,GS21}.
We then define a family of flows for left-invariant Dorfman brackets on an exact Courant algebroid over a simply connected nilpotent Lie group, generalizing the bracket flows for nilpotent Lie brackets in a way that might make this new family of flows useful for the study of generalized geometric flows, such as the generalized Ricci flow.
We provide explicit examples of both constructions on the Heisenberg group. We also discuss solutions to the generalized Ricci flow on the Heisenberg group.
\end{abstract}

\maketitle


\section{Introduction}
Generalized geometry, building on the work of N.\ Hitchin \cite{Hit03} and M.\ Gualtieri \cite{Gua04} and the structure of Courant algebroids, constitutes a rich mathematical environment.
The main idea behind it lies in the shift of point of view when studying structures on a differentiable manifold $M^n$, replacing the tangent bundle $TM$ with the \emph{generalized tangent bundle}
\[
\mathbb{T}M=TM \oplus T^*M.
\]
More explicitly, in the language of $G$-structures, one studies reductions of $\text{GL}(\mathbb{T}M)$, the $\text{GL}_{2n}$-principal bundle of frames of $\mathbb{T}M$.

A reduction to the orthogonal group $\text{O}(n,n)$ always exists, thanks to the nondegenerate symmetric bilinear form of neutral signature
\begin{equation} \label{scal}
\left<X+\xi,Y+\eta\right>=\frac{1}{2}\left(\eta(X)+\xi(Y)\right),
\end{equation}
so that one usually only considers structures which are reductions of $\text{O}(\mathbb{T}M)$, the $\text{O}(n,n)$-reduction of $\text{GL}(\mathbb{T}M)$ determined by this natural pairing.

In this spirit, for example, a \emph{generalized almost complex structure} on $M^{2m}$, defined by an orthogonal automorphism $\mathcal{J}$ of $\mathbb{T}M$, $\mathcal{J}^2=-\operatorname{Id}_{\mathbb{T}M}$, determines a $\text{U}(m,m)$-reduction of $\text{GL}(\mathbb{T}M)$. The integrability of such a structure is expressed through an involutivity condition with respect to a natural bracket operation, called the \emph{Dorfman bracket}:
\begin{equation} \label{brack}
\left[ X+\xi, Y+\eta\right]=[X,Y]+\mL_X \eta - \iota_Y d\xi.
\end{equation}

On the other hand, a \emph{generalized Riemannian metric} on $M^n$, defined by a symmetric (with respect to $\left<\cdot,\cdot\right>$) and involutive automorphism $\G$ of $\mathbb{T}M$, determines an $\text{O}(n) \times \text{O}(n)$-reduction of $\text{GL}(\mathbb{T}M)$.

More generally, one can consider a \emph{Courant algebroid} $E$ over $M$, namely a smooth vector bundle over $M$ endowed with a pairing $\left<\cdot,\cdot\right>$ and a bracket $\left[ \cdot,\cdot \right]$ satisfying certain properties so that $\mathbb{T}M$, endowed with \ref{scal} and \ref{brack}, is a special case. On the generic Courant algebroid one can then study reductions of $\text{GL}(E)$, such as generalized almost complex structures and generalized \mbox{(pseudo-)}Riemannian metrics.

In \cite{Str17,Gar19,GS21}, the authors introduced a flow of generalized (pseudo-)Riemannian metrics on a Courant algebroid $E$ over a smooth manifold $M$, generalizing the classical Ricci flow of R. Hamilton \cite{Ham82} and the $B$-field renormalization group flow of Type II string theory (see \cite{Pol98}). The \emph{generalized Ricci flow}, as we shall refer to this flow from now on, is actually a flow for a pair of families of generalized (pseudo-)Riemannian metrics $\G \in \text{Aut}(E)$ and divergence operators $\text{div} \colon \Gamma(E) \to C^{\infty}(M)$, the latter of which are required in order to ``gauge-fix'' curvature operators associated with a generalized  (pseudo-)Riemannian metric.

The paper is organized as follows: Section \ref{sec:preliminaries} is devoted to a review of the setting of generalized geometry -- including the notions of Courant algebroid, generalized curvature tensors and the definition of generalized Ricci flow -- and of the algebraic framework of nilpotent Lie groups.

In Section \ref{sec:soliton} we introduce the notion of \emph{generalized Ricci soliton}, which derives from the study of self-similar (in a suitable sense) solutions to the generalized Ricci flow on exact Courant algebroids. This condition generalizes the Ricci soliton condition $\text{Rc}_g=\lambda g + \mL_X g,$ where $\text{Rc}_g$ denotes the Ricci tensor of $g$, $\lambda \in \R$ and $\mL_X g$ denotes the Lie derivative of $g$ with respect to a vector field $X$. We show that, when working on a Lie group and considering left-invariant structures, this condition descends to an algebraic condition on the Lie algebra of the group.

Borrowing from the ideas of J. Lauret, in Section \ref{sec:bracket} we consider left-invariant Dorfman brackets on simply connected nilpotent Lie groups, describing them as elements of an algebraic subset of the vector space of skew-symmetric bilinear forms on \mbox{$\rr$}, for the suitable $n$. We then define a family of flows of such structures, showing that they generalize the constructions known in literature as \emph{bracket flows}, which have been extensively used to rephrase geometric flows on (nilpotent) Lie groups (see for example \cite{Lau11}). This justifies our definition of \emph{generalized bracket flows}.

In Section \ref{examples_heis}, we perform explicit computations of generalized Ricci solitons and exhibit an example of generalized bracket flow on the three-dimensional Heisenberg group.

In Section \ref{genricciflow_heis}, we study solutions of the generalized Ricci flow on the Heisenberg group, highlighting the differences with the classical Ricci flow.
\medskip

\noindent
{\bf Acknowledgments.}
This paper is an adaptation of the author's master's thesis, written under the supervision of Anna Fino. To her the author wishes to express his most sincere gratitude. The author also wishes to thank Mario Garcia-Fernandez for useful comments and Jeffrey Streets for pointing out reference \cite{Str17}. He also thanks David Krusche for noting an imprecision in formula \ref{Rc_gen}, and an anonymous referee for useful comments which helped improve the presentation of the paper. The author was supported by GNSAGA of INdAM.

\section{Preliminaries}\label{sec:preliminaries}

\subsection{Courant algebroids}

Let $V$ be a real vector space of dimension $n$.
We start by recalling a few facts about the algebra of the vector space $V \oplus V^*$; for more details, see \cite{Gua04}.

$V \oplus V^*$ can be endowed with a natural symmetric bilinear form of neutral signature
\[
\left<X+\xi,Y+\eta\right>=\frac{1}{2}(\eta(X)+\xi(Y))
\]
and with a canonical orientation provided by the preimage of $1 \in \R$ in the isomorphism $\varphi \colon \Lambda^{2n}(V \oplus V^*) = \Lambda^nV \otimes \Lambda^nV^*  \to \R$, sending $(X_1 \wedge \ldots \wedge X_n) \otimes (\xi_1 \wedge \ldots \wedge \xi_n)$ into $\text{det} (\xi_i(X_j))_{ij}$.

Consider the Lie group $\text{SO}(V \oplus V^*) \cong \text{SO}(n,n)$ of automorphisms of $V \oplus V^*$ preserving the pairing $\left< \cdot,\cdot \right>$ and the canonical orientation. Its Lie algebra $\mathfrak{so}(V \oplus V^*) \cong \mathfrak{so}(n,n)$ consists of  endomorphisms $T \in \mathfrak{gl}(V \oplus V^*)$ which are skew-symmetric with respect to $\left<\cdot,\cdot\right>$, namely
\begin{equation} \label{skewsymm}
\left<Tz_1,z_2\right> + \left<z_1,Tz_2\right>=0
\end{equation}
for all $z_1,z_2 \in V \oplus V^*$. Seeing $T$ as a block matrix, \ref{skewsymm} dictates $T$ to be of the form
\[
T=\begin{pmatrix} \phi & \beta \\ B & -\phi^* \end{pmatrix},
\]
for some $\phi \in \mathfrak{gl}(V)$, $B \in \Lambda^2V^*$ and $\beta \in \Lambda^2V$, recovering the fact that
\[
\mathfrak{so}(V \oplus V^*) \cong \Lambda^2(V \oplus V^*)^* \cong \Lambda^2 V^* \oplus (V^* \otimes V) \oplus \Lambda^2 V,
\]
where the former isomorphism is given by $T \mapsto \left<T \cdot,\cdot\right>$.

Via the exponential map $\text{exp} \colon \mathfrak{so}(V \oplus V^*) \to \text{SO}(V \oplus V^*)$, we obtain distinguished elements of $\text{SO}(V \oplus V^*)$:
\begin{itemize}
\item $e^B = \begin{pmatrix} \operatorname{Id} & 0 \\ B & \operatorname{Id} \end{pmatrix} \colon X + \xi \mapsto X+\xi+\iota_X B$, called \emph{$B$-field transformations},
\item $e^\phi= \begin{pmatrix} e^\phi & 0 \\ 0 & (e^{-\phi})^* \end{pmatrix}$, which extends to an embedding of the whole $\text{GL}(V)$ into \mbox{$\text{SO}(V \oplus V^*)$}, sending $A \in \text{GL}(V)$ into
\[
\bm{A}=\begin{pmatrix} A & 0 \\0 & (A^*)^{-1} \end{pmatrix}.
\]
In the case $V=\R^n$, the image of this embedding will be denoted by $\textbf{GL}_n$.
\end{itemize}

Let $M$ be an oriented smooth manifold of positive dimension $n$.
\begin{definition}
A \emph{Courant algebroid} over $M$ is a smooth vector bundle $E \rightarrow M$ equipped with:
\begin{itemize}
\item $\left<\cdot,\cdot\right>$, a fiberwise nondegenerate bilinear form, which allows to identify $E$ and its dual $E^*$, viewing $z \in E$ as $\left<z,\cdot\right> \in E^*$,
\item $[\cdot,\cdot]$, a bilinear operator on $\Gamma(E)$,
\item a bundle homomorphism $\pi : E \rightarrow TM$, called the \emph{anchor},
\end{itemize}
which satisfy the following properties for all $z,z_i \in \Gamma(E)$, $i=1,2,3$, $f \in C^\infty(M)$:
\begin{enumerate}
\item $[z_1,[z_2,z_3]]=[[z_1,z_2],z_3]+[z_2,[z_1,z_3]]$ (Jacobi identity),
\item $\pi[z_1,z_2]=[\pi(z_1),\pi(z_2)]$,
\item $[z_1,fz_2]=f[z_1,z_2]+\pi(z_1)(f)z_2$,
\item $[z,z]=\frac{1}{2}\D\left<z,z\right>$,
\item $\pi(z_1)\left<z_2,z_3\right>=\left<[z_1,z_2],z_3\right>+\left<z_2,[z_1,z_3]\right>$,
\end{enumerate}
where $\D \coloneqq \pi^*d:C^{\infty}(M) \rightarrow \Gamma(E)$.
\end{definition}

\begin{definition}
A Courant algebroid $E$ over $M$ is \emph{exact} if the short sequence
\begin{equation} \label{exact}
\begin{tikzcd}
0\arrow{r} & T^*M \arrow{r}{\pi^*} & E\arrow{r}{\pi} & TM\arrow{r} & 0
\end{tikzcd}
\end{equation}
is exact, namely if the anchor map is surjective and its kernel is exactly the image of $\pi^*$.
\end{definition}
By the classification of P. \v{S}evera \cite{Sev98}, isomorphism classes of exact Courant algebroids over $M$ are in bijection with the elements of the third de Rham cohomology group of $M$, $H^3(M)$: an exact Courant algebroid with \emph{\v{S}evera class} $[H] \in H^3(M)$ is isomorphic to the Courant algebroid $E_H=\mathbb{T}M \coloneqq TM \oplus T^*M$ over $M$ with pairing of neutral signature
\begin{equation} \label{innerstd}
\left<X+\xi,Y+\eta\right>=\frac{1}{2}(\eta(X)+\xi(Y)),
\end{equation}
and \emph{(twisted) Dorfman bracket}
\begin{equation} \label{bracketH}
[X+\xi,Y+\eta]=[X,Y]+\mL_X \eta - \iota_Y d\xi + \iota_Y \iota_X H,
\end{equation}
for any $H \in [H]$. Such isomorphisms are obtained explicitly via the choice of an isotropic splitting to \ref{exact}, while $B$-field transformations, $B \in \Gamma(\Lambda^2T^*M)$, provide explicit isomorphisms
\[
e^B \colon E_H \to E_{H-dB}.
\]

In what follows, let $E$ be a Courant algebroid over $M$, with $\text{rk}(E)=2n$ and pairing $\left< \cdot,\cdot \right>$ of neutral signature.
\begin{definition}
A \emph{generalized Riemannian metric} on $E$ is an $\text{O}(n) \times \text{O}(n)\,$-reduction of $\text{O}(E)$, the $\text{O}(n,n)$-principal subbundle of orthonormal frames of $E$ with respect to the pairing $\left<\cdot,\cdot\right>$. Explicitly, it is equivalently determined by
\begin{itemize}
\item a subbundle $E_+$ of $E$, $\text{rk}(E_+)=n$, on which $\left<\cdot,\cdot\right>$ is positive-definite,
\item an automorphism $\G$ of $E$ which is involutive, namely $\G^2=\operatorname{Id}_E$, and such that $\left< \G \cdot,\cdot\right>$ is a positive-definite metric on $E$.
\end{itemize}
Given $E_+$, denoting by $E_-$ its orthogonal complement with respect to $\left<\cdot,\cdot\right>$, $\G$ is defined by $\G|_{E_\pm}=\pm \operatorname{Id}_{E_\pm}$. $E_\pm$ can then be recovered as the $\pm 1$-eigenbundles of $\G$.
Given $z \in E$, we shall denote by $z^\pm$ its orthogonal projections along $E_\pm$.
\end{definition}

\begin{example}
Every generalized Riemannian metric on the exact Courant algebroid $E_H$ is of the form
\begin{equation*}
\G=e^B \begin{pmatrix} 0 & g^{-1} \\ g & 0 \end{pmatrix} e^{-B},
\end{equation*}
for some $g$ Riemannian metric and $B$ $2$-form on $M$ (see \cite[Section 6.2]{Gua04}). The corresponding $E_\pm$ are
\[
E_\pm=e^B \{X\pm g(X),\,X \in TM \},
\]
where by $g(X)$ we mean $g(X,\cdot)$.
Notice that $\G$ is of the form
\begin{equation*}
\begin{pmatrix} 0 & g^{-1} \\ g & 0 \end{pmatrix}
\end{equation*}
in the splitting $E_{H+dB}$.
\end{example}

\subsection{Generalized curvature}
We now recall the definition of generalized connection on a Courant algebroid $E$, showing how these objects can be used to associate curvature operators with a generalized Riemannian metric $\G$.
Unlike the Riemannian case, where the uniqueness of the Levi-Civita connection allows to single out canonical curvature operators for a given Riemannian metric, in the generalized setting there are plenty of torsion-free generalized connections compatible with a generalized Riemannian metric $\G$, and these may define different curvature operators. To gauge-fix them, one needs to additionally fix a divergence operator. For further details, we refer the reader to \cite{Gar19} and \cite{CD19}.
\begin{definition}
A \emph{generalized connection} on a Courant algebroid $E$ is a linear map
\[
D \colon \Gamma(E) \to \Gamma(E^* \otimes E)
\]
which satisfies a Leibniz rule and a compatibility condition with $\left<\cdot,\cdot\right>$:
\begin{gather*}
D(fz)=f(D z) + \D f \otimes z,\\
\left<\D \left<z_1,z_2\right>,\cdot\right>=\left<D_\cdot z_1,z_2\right> + \left<z_1,D_\cdot z_2\right>,
\end{gather*}
for all $z,z_1,z_2 \in \Gamma(E)$, $f \in C^{\infty}(M)$, where $D_{z_1} z_2 \coloneqq D z_2(z_1)$.
\end{definition}

Given a generalized Riemannian metric $\G$, a generalized connection $D$ is \emph{compatible with $\G$} if $D\G=0$, where here $D$ denotes the induced $E$-connection on the tensor bundle $E^* \otimes E \cong \text{End}(E)$. Equivalently, $D$ is compatible with $\G$ if $D(\Gamma(E_\pm)) \subset \Gamma(E^* \otimes E_\pm)$.

The \emph{torsion} $T_D \in \Gamma(\Lambda^2E^* \otimes E)$ of a generalized connection $D$ on E is defined by
\[
T_D(z_1,z_2)=D_{z_1}z_2 - D_{z_2}z_1 - [z_1,z_2] + (D z_1)^*z_2.
\]
If $T_D=0$, the generalized connection $D$ is said to be \emph{torsion-free}.

Given a generalized connection $D$ on $E$ which is compatible with a generalized Riemannian metric $\G$, one can define curvature operators
\[
\text{R}_D^\pm \in \Gamma(E_{\pm}^* \otimes E_{\mp}^* \otimes \mathfrak{o}(E_\pm)),
\]
where $\mathfrak{o}(E_\pm)=\left<\cdot,\cdot\right>^{-1} \Lambda^2 E_\pm^*$ denotes the Lie algebra of skew-symmetric endomorphisms of $E_\pm$ with respect to $\left<\cdot,\cdot\right>$, by
\[
\text{R}_D^{\pm}\left(z_1^\pm,z_2^\mp\right)z_3^\pm=D_{z_1^\pm}D_{z_2^\mp}z_3^\pm - D_{z_2^\mp}D_{z_1^\pm}z_3^\pm - D_{\left[z_1^\pm,z_2^\mp\right]}z_3^\pm.
\]
One then has associated \emph{Ricci tensors}
\begin{align*}
\text{Rc}^\pm_D &\in \Gamma(E_\mp^* \otimes E_\pm^*),\\
\text{Ric}^\pm_D &\in \Gamma(E_\mp^* \otimes E_\pm^{}),
\end{align*}
defined by
\begin{gather*}
\text{Rc}_D^\pm(z_1^\mp,z_2^\pm) = \text{tr}\left(z^\pm \mapsto \text{R}^\pm_D(z^\pm,z_1^\mp)z_2^\pm\right),\\
\text{Ric}_D^\pm=\left<\cdot,\cdot\right>^{-1}\text{Rc}_D^\pm.
\end{gather*}

\begin{definition}
A \emph{divergence operator} on $E$ is a first order differential operator $\text{div} \colon \Gamma(E) \to C^{\infty}(M)$ satisfying the Leibniz rule
\[
\text{div}(fz)=\pi(z)f + f\,\text{div}(z),
\]
$f \in C^{\infty}(M)$, $z \in \Gamma(E)$.
Given a generalized connection $D$ on $E$, one may define the associated divergence operator
\[
\text{div}_D(z)=\text{tr}(Dz).
\]
\end{definition}

\begin{remark}
Divergence operators on $E$ form an affine space over the vector space $\Gamma(E) \cong \Gamma(E^*)$. Fixing a divergence operator $\text{div}_0$, any other $\text{div}$ is of the form
\[
\text{div}=\text{div}_0 - \left<z,\cdot\right>
\]
for some $z \in \Gamma(E)$.
\end{remark}

\begin{proposition}{\normalfont \cite[Proposition 4.4]{Gar19}}
Let $D_i$, $i=1,2$, be torsion-free generalized connections on $E$ compatible with a given generalized Riemannian metric $\G$. Suppose $\text{\normalfont div}_{D_1}=\text{\normalfont div}_{D_2}$. Then, $\text{\normalfont Rc}^\pm_{D_1}=\text{\normalfont Rc}^\pm_{D_2}$.
\end{proposition}

Moreover, for any divergence operator $\text{div}$ and generalized Riemannian metric $\G$ on $E$, the set of torsion-free generalized connections $D$ on $E$ which are compatible with $\G$ and such that $\text{div}_D =\text{div}$ is nonempty (see \cite[Section 3.2]{Gar19}). Thus, Ricci tensors $\text{Rc}_{\G,\text{div}}^\pm$ are well-defined as equal to $\text{Rc}_D^\pm$ for any such generalized connection $D$.

\begin{example}
On the exact Courant algebroid $E_H$ over $M$, let
\[
\G=\begin{pmatrix} 0 & g^{-1} \\ g & 0 \end{pmatrix}
\]
and
\[
\text{div}_{g,z}(X+\xi)=dV_g^{-1}\mL_X dV_g^{} - \left<z,X+\xi\right>,
\]
where $g$ is a Riemannian metric, $dV_g$ its associated Riemannian volume form and $z \in \Gamma(E_H)$. Then, via the isomorphism $\pi_+ = \pi|_{E_+} \colon E_+ \to TM$, the Ricci tensor $\text{Rc}^+$ of $(\G,\text{div}_{g,\theta})$ is given by
\begin{equation} \label{Rc_gen}
\text{Rc}^+_{\G,\text{div}_{g,z}}=\text{Rc}_g - \frac{1}{4} H \upperset{g}{\circ} H - \frac{1}{2} d^*_g H + \frac{1}{2} \nabla^{+}_{g,H} \theta,
\end{equation}
where
\begin{itemize}
\item $\text{\normalfont Rc}_g \in \Gamma(S^2_+T^*M)$ is the Ricci tensor associated with $g$,
\item $H \upperset{g}{\circ} H \in \Gamma(S^2T^*M)$,
\[
H \upperset{g}{\circ} H (X,Y) = g(\iota_X H,\iota_Y H),
\]
\item $d^*_g=-\upperset{g}{*} d \,\upperset{g}{*} \colon \Gamma(\Lambda^3T^*M) \to \Gamma(\Lambda^2T^*M)$ is the Hodge codifferential associated with the metric $g$ and the fixed orientation, $\upperset{g}{*}$ being the Hodge star operator,
\item $\nabla^+_{g,H} = \nabla^g + \frac{1}{2} g^{-1}H$ is the \emph{Bismut connection} with torsion $H$, $\nabla^g$ denoting the Levi-Civita connection of $g$,
\item $\theta \in \Gamma(T^*M)$ is given by $\theta=2g(\pi z^+,\cdot)=g(X,\cdot)+\xi$, if $z=X+\xi$.
\end{itemize}
See \cite[Proposition 3.30]{GS21} for the proof of this fact (cf. also \cite{Kru}).
\end{example}

\subsection{Generalized Ricci flow}

We now review the framework of the generalized Ricci flow first introduced in \cite{Str17,Gar19} and later described and studied in \cite{GS21} by the two authors. Consider a smooth family of generalized Riemannian metrics $(\G(t))_{t \in I}$ on $E$, $I \subset \R$, with respective eigenbundles ${E_\pm}|_t$. Its variation $\dot{\G}(t)$ exchanges the eigenbundles ${E_{\pm}}|_t$, so that $\dot{\G}(t)=\dot{\G}^+(t) + \dot{\G}^-(t)$, with
\[
\dot{\G}^\pm(t) \in \Gamma({E_{\mp}}|_t^* \otimes {E_{\pm}}|_t^{}).
\]

\begin{definition}{\cite[Definition 5.1]{Gar19}}
A smooth pair of families $(\G(t),\text{div}(t))_{t \in I}$ of generalized Riemannian metrics and divergence operators on $E$ is a solution to the \emph{generalized Ricci flow} if it satisfies
\[
\dot{\G}^+(t)=-2\,\text{Ric}_t^+,
\]
for all $t$ in the interior of $I$, where $\text{Ric}_t^+ \coloneqq \text{Ric}^+_{\G(t),\text{div}(t)}$.
\end{definition}

On an exact Courant algebroid, the system may be written as follows:

\begin{proposition}{\normalfont \cite[Example 5.4]{Gar19}} \label{prop:GRFex}
Let $E$ be an exact Courant algebroid on an oriented smooth manifold $M$, with \v{S}evera class $[H] \in H^3(M)$. Fix an isotropic splitting $E_{H}=\mathbb{T}M$ for $E$ and consider the pair of smooth families $(\G(t),\text{\normalfont div}(t))_{t \in I}$ defined by:
\begin{gather*}
\G(t)=e^{B(t)} \begin{pmatrix} 0 & g(t)^{-1} \\ g(t) & 0  \end{pmatrix} e^{-B(t)}, \\
\text{{\normalfont div}}(t)=\text{{\normalfont div}}_{g(t),z(t)},
\end{gather*}
where $(g(t)) \subset \Gamma(S^2_+T^*M)$, $(B(t)) \subset \Gamma(\Lambda^2T^*M)$ and $(z(t)) \subset \Gamma(E)$.

Then $(\G(t),\text{\normalfont div}(t))_{t \in I}$ is a solution of the generalized Ricci flow on $E$ if and only if the families $(g(t),B(t),\theta(t))_{t \in I}$, with $\theta(t)=2g(\pi z(t)^+,\cdot) \in \Gamma(T^*M)$, solve the equation
\begin{equation} \label{eq:GRFex}
\dot{g}(t)=-2\,\left(\text{\normalfont Rc}_{g(t)} - \frac{1}{4}H(t) \upperset{g(t)}{\circ} H(t) - \frac{1}{2} d_{{}^{g(t)}}^*H(t) + \frac{1}{2} {\nabla_{g(t),H(t)}^+} \theta(t) \right)+\dot{B}(t),
\end{equation}
where $H(t)=H+dB(t)$.
\end{proposition}

Separating the symmetric and skew-symmetric part of \ref{eq:GRFex} one gets (see \cite{ST13})
\[
\begin{dcases*}
\dot{g}(t)= -2\,\text{Rc}_{g(t)} + \frac{1}{2} H(t) \upperset{g(t)}{\circ} H(t) - \frac{1}{2} \mL_{g(t)^{-1}\theta(t)}g(t), \\
\dot{B}(t)= -d^*_{g(t)}H(t) + \frac{1}{2} d\theta(t) - \frac{1}{2} \iota_{g(t)^{-1}\theta(t)} H(t),
\end{dcases*}
\]
where one has that
\[
\frac{1}{2}\mL_{g(t)^{-1}\theta(t)}g(t)= S({\nabla_{g(t),H(t)}^+} \theta(t)), \qquad \frac{1}{2} d\theta(t) - \frac{1}{2} \iota_{g(t)^{-1}\theta(t)}H(t)= A({\nabla_{g(t),H(t)}^+} \theta(t))
\]
are respectively the symmetric and skew-symmetric parts of ${\nabla_{g(t),H(t)}^+} \theta(t)$.

The pair $(g(t),H(t))$ evolves as
\begin{equation} \label{eq:GRF_gH}
\begin{dcases*}
\dot{g}(t)= -2\,\text{Rc}_{g(t)} + \frac{1}{2} H(t) \upperset{g(t)}{\circ} H(t) - \frac{1}{2} \mL_{g(t)^{-1}\theta(t)}g(t), \\
\dot{H}(t)= -\Delta_{g(t)}H(t)-\frac{1}{2} \mL_{g(t)^{-1}\theta(t)} H(t),
\end{dcases*}
\end{equation}
where $\Delta_g=dd^*_g + d^*_g d$ denotes the Hodge Laplacian operator associated with $g$ and the fixed orientation. Notice how, up to scaling, the \emph{pluriclosed flow} introduced in \cite{ST10} is equivalent to a particular case of the generalized Ricci flow, as is proven in Propositions 6.3 and 6.4 in \cite{ST13}.
By \cite[Theorem 6.5]{ST13} a solution to \ref{eq:GRF_gH} can be pulled back to a solution of
\begin{equation}
\label{GRF_gauge}
\begin{dcases*}
\dot{g}(t)= -2\,\text{Rc}_{g(t)} + \frac{1}{2} H(t) \upperset{g(t)}{\circ} H(t), \\
\dot{H}(t)= -\Delta_{g(t)}H(t),
\end{dcases*}
\end{equation}
via the one-parameter family of diffeomorphism generated by $\frac{1}{4}g(t)^{-1}\theta(t)$.

\subsection{Simply connected nilpotent Lie groups}

We briefly recall the structure of simply connected nilpotent Lie groups, in the description of J. Lauret (see for example \cite{Lau11}).

Every simply connected nilpotent Lie group $G$ is diffeomorphic to its Lie algebra of left-invariant fields $\mathfrak{g}$ via the exponential map. Identifying $\mathfrak{g}$ with $\mathbb{R}^n$ via the choice of a basis, denote by $\mu \in \Lambda^2(\mathbb{R}^n)^* \otimes \mathbb{R}^n$ the induced Lie bracket. Now, exploiting the Campbell-Baker-Hausdorff formula,
\[
\text{exp}(X)\cdot \text{exp}(Y) = \text{exp}(X+Y+p_\mu(X,Y)),
\]
$X,Y \in \mathfrak{g} \cong \mathbb{R}^n$, where $p_\mu$ is a $\mathbb{R}^n$-valued polynomial in the variables $X,Y$, one can endow $\mathbb{R}^n$ with the operation $\cdot_\mu$,
\[
X \cdot_\mu Y = X + Y + p_\mu(X,Y),
\]
so that $\text{exp} \colon (\mathbb{R}^n,\cdot_\mu) \to G$ is an isomorphism of Lie groups. Therefore, the set of isomorphism classes of simply connected nilpotent Lie groups is parametrized by the set of nilpotent Lie brackets on $\R^n$: these form an algebraic subset of the vector space of skew-symmetric bilinear forms on $\R^n$,
\[
\mathcal{V}_n \coloneqq \Lambda^2 (\R^n)^* \otimes {\R^n},
\]
which parametrizes all skew-symmetric algebra structures on ${\R^n}$. Coordinates for $\mathcal{V}_n$ can be obtained by fixing a basis $\{e_i\}_{i=1}^n$ for ${\R^n}$: this allows to determine the so-called \emph{structure constants} of any fixed $\mu \in \mathcal{V}_n$ as the real numbers $\{\mu_{ij}^k,\,i,j,k=1 \dots n\}$ given by
\[
\mu(e_i,e_j)=\mu_{ij}^k e_k.
\]

One can then consider
\[
\mL_n\coloneqq\{\mu \in \mathcal{V}_n,\,\mu \text{ satisfies the Jacobi identity}\},
\]
the algebraic subset of $\mathcal{V}_n$ consisting of Lie brackets on ${\R^n}$, and
\[
\mathcal{N}_n \coloneqq \{\mu \in \mL_n,\,\mu \text{ is nilpotent}\},
\]
which parametrizes all nilpotent Lie algebra structures on ${\R^n}$.
By the previous remarks, $\mathcal{N}_n$ parametrizes all $n$-dimensional simply connected nilpotent Lie groups, up to isomorphism.

Let us consider the family of Riemannian metrics on $\R^n$
\begin{equation} \label{metrics}
\{ g_{\mu,q},\,\mu \in \mathcal{N}_n,\,q \text{ positive definite bilinear form on } \R^n \},
\end{equation}
where $g_{\mu,q}$ coincides with $q$ at the origin and is left-invariant with respect to the nilpotent Lie group operation $\cdot_\mu$. The set \ref{metrics} is actually the set of all Riemannian metrics on $\R^n$ which are invariant by some transitive action of a nilpotent Lie group. By \cite[Theorem 3]{Wil82}, the Riemannian manifolds $(\R^n,g_{\mu,q})$ (varying $n$, $\mu$ and $q$) are, up to isometry, all the possible examples of simply connected \emph{homogeneous nilmanifolds}, namely connected Riemannian manifolds admitting a transitive nilpotent Lie group of isometries.

The Riemannian metrics in \ref{metrics} are not all distinct, up to isometry: it was shown again in \cite[Theorem 3]{Wil82} that $g_{\mu,q}$ is isometric to $g_{\mu^\prime,q^\prime}$ if and only if there exists $h \in \text{GL}_n$ such that $\mu^\prime = h^*\mu$ and $q^\prime=h^*q$.
By convention we shall denote $g_\mu\coloneqq g_{\mu,\left<\cdot,\cdot\right>}$, where $\left<\cdot,\cdot\right>$ denotes the standard scalar product.

Since the Riemannian metrics $g_{\mu,q}$ are completely determined by their value at $0$ and by the Lie bracket $\mu$, so will be all curvature quantities related to $g_{\mu,q}$.
In particular, we are interested in Riemannian metrics $g_\mu$ and their Ricci tensor, which we shall encounter in two guises, which we denote by
\begin{align*}
&\rc_\mu\coloneqq \rc_{g_\mu}(0) \in S^2(\R^n)^* \subset (\R^n)^* \otimes (\R^n)^*,\\
&\ric_\mu\coloneqq \ric_{g_\mu}(0) \in (\R^n)^* \otimes \R^n = \mathfrak{gl}_n,
\end{align*}
with $\rc_\mu(X,Y)=\left<\ric_\mu(X),Y\right>$, $X,Y \in \R^n$.

For these, explicit formulas can be computed \cite{Lau01}. Let $\{e_i\}_{i=1}^n$ be the standard basis of $\R^n$, which, in particular, is orthonormal with respect to $\left<\cdot,\cdot\right>$: one has
\begin{equation} \label{riccinil_form}
\begin{aligned}
\rc_\mu(X,Y)=-\frac{1}{2} \left<\mu(X,e_k),e_l\right>\left<\mu(Y,e_k),e_l\right> + \frac{1}{4} \left<\mu(e_k,e_l),X\right>\left<\mu(e_k,e_l),Y\right>,
\end{aligned}
\end{equation}
so that, if $\rc_\mu=(\rc_\mu)_{ij}\,e^i \otimes e^j$ and $\ric_\mu = \left( \ric_\mu \right)_i^j e^i \otimes e_j$, one has
\begin{equation}\label{riccinil_coord}
(\rc_\mu)_{ij}=(\ric_\mu)_i^j=-\frac{1}{2} \mu_{ik}^l\mu_{jk}^l + \frac{1}{4} \mu_{kl}^i \mu_{kl}^j.
\end{equation}
Notice that one can use formulas \ref{riccinil_form} and \ref{riccinil_coord} to define $\rc_\mu \in S^2(\R^n)^*$ and $\ric_\mu \in \mathfrak{gl}_n$ for any $\mu \in \mathcal{V}_n$. 

\section{Generalized Ricci solitons} \label{sec:soliton}

Just as Ricci soliton metrics arise from self-similar solutions of the Ricci flow, generalized Ricci solitons arise from self-similar solutions of the generalized Ricci flow. We focus on exact Courant algebroids, defining a family of generalized Riemannian metrics, whose initial one is determined by a Riemannian metric on the base manifold; imposing that this family (together with a family of divergence operators) is a solution of the generalized Ricci flow, we draw necessary conditions on said Riemannian metric: these conditions generalize the Ricci soliton condition, leading to the definition of \emph{generalized Ricci solitons}.

Let $E$ be a Courant algebroid over an oriented smooth manifold $M$ with \v{S}evera class $[H_0] \in H^3(M)$. Fixing an isotropic splitting $E_{H_0}=\mathbb{T}M$, we consider a smooth self-similar pair of families $(\G(t),\text{div}(t))_{t \in I}$, $0 \in I$, on $E$ of the form
\begin{gather*}
\G(t)=e^{B(t)} \begin{pmatrix} 0 & (c(t)\varphi_t^*g_0)^{-1} \\ c(t)\varphi_t^*g_0 & 0  \end{pmatrix} e^{-B(t)}, \\
\text{{\normalfont div}}(t)=\text{{\normalfont div}}_{g(t),\theta(t)},
\end{gather*}
where $g_0 \in \Gamma(S^2_+(T^*M))$ is a Riemannian metric, $c:I \to \mathbb{R}$ is smooth and positive, $c(0)=1$, $(\varphi_t)$ is a one-parameter family of diffeomorphisms of $M$, $(B(t)) \subset \Gamma(\Lambda^2T^*M)$, $B(0)=0$, $(\theta(t)) \subset \Gamma(T^*M)$, $\theta(0)=\theta_0 \in \Gamma(T^*M)$ and $g(t)=c(t)\varphi_t^*g_0$.

By Proposition \ref{prop:GRFex}, such $(\G(t),\text{div}(t))_{t \in I}$ is a solution of the generalized Ricci flow if and only if
\begin{equation*}
\begin{dcases} \dot{c}(t)\varphi_t^*g_0+c(t)\varphi_t^*\mL_{Y_t}g_0=-2\,\text{\normalfont Rc}_{g(t)} + \dfrac{1}{2}H(t) \upperset{g(t)}{\circ} H(t) + \dfrac{1}{4} \mL_{g(t)^{-1}\theta(t)}g(t),\\
\dot{B}(t)=-d_{{}^{g(t)}}^*H(t) - \dfrac{1}{4} d\theta(t) + \frac{1}{4}\iota_{g(t)^{-1}\theta(t)}H(t),
\end{dcases}
\end{equation*}
where $H(t)=H_0 +dB(t)$ and $(Y_t)_{t \in I} \subset \Gamma(TM)$ is such that
\[
\frac{d}{dt}\varphi_t(x)=Y_t(\varphi_t(x)),
\]
for all $t \in I$, $x \in M$.

Setting $t=0$ and rearranging the terms,
\begin{equation} \label{genriccisol_split}
\begin{dcases} 
\text{Rc}_{g_0} = \lambda g_0 + \mL_X g_0 + \dfrac{1}{4} H_0 \upperset{g_0}{\circ} H_0 - \dfrac{1}{4} \mL_{g_0^{-1}\theta_0} g_0, \\
\omega= -d_{{}^{g_0}}^*H_0 + \frac{1}{2}d\theta_0 - \frac{1}{2} \iota_{g_0^{-1}\theta_0}H_0,
\end{dcases}
\end{equation}
where $-2\lambda = \dot{c}(0) \in \mathbb{R}$, $-2X=Y_0 \in \Gamma(TM)$, $\omega=\dot{B}(0) \in \Gamma(\Lambda^2T^*M)$. Summing together the two equations of \ref{genriccisol_split}, which involve symmetric and skew-symmetric tensor fields respectively, one has
\begin{equation} \label{genriccisol}
\text{Rc}_{g_0} = \lambda g_0 + \mL_X g_0 + \frac{1}{4} H_0 \upperset{g_0}{\circ} H_0 - \frac{1}{2} {\nabla_{g_0,H_0}^+} \theta_0 + \frac{1}{2}d_{{}^{g_0}}^*H_0 + \frac{1}{2} \omega,
\end{equation}
which is therefore equivalent to \ref{genriccisol_split}.
We can now introduce the following definition, which generalizes the notion of Ricci soliton.

\begin{definition}
A Riemannian metric $g_0$ on $M$ is called a \emph{generalized Ricci soliton} if there exist $\lambda \in \mathbb{R}$, $X \in \Gamma(TM)$, $H_0 \in \Gamma(\Lambda^3T^*M)$ closed, $\theta_0 \in \Gamma(T^*M)$, $\omega \in \Gamma(\Lambda^2T^*M)$ such that \ref{genriccisol}, or equivalently \ref{genriccisol_split}, holds.
\end{definition}

When working on a Lie group $G$, for simplicity one can assume all structures to be left-invariant, so that the generalized Ricci soliton condition reduces to an algebraic condition on structures on the Lie algebra of $G$, $(\mathfrak{g},\mu)$.

In the context of semi-algebraic Ricci solitons, it was proven in \cite[Theorem 1.5]{Jab15} that, if $g_0$ is a left-invariant Riemannian metric on $G$, the Lie derivative of $g_0$ with respect to a left-invariant vector field $X$ can be written as
\[
\mL_Xg_0=g_0(\tfrac{1}{2}(D+D^t))=g_0(\tfrac{1}{2}(D+D^t) \cdot, \cdot),
\]
for some $D=D_X \in \text{Der}(\mathfrak{g})$, where $\text{Der}(\mathfrak{g})$ denotes the algebra of derivations of $\mathfrak{g}$. It was then shown in \cite[Theorem 1]{Jab14} (generalizing the already known fact for the simply connected nilpotent case in \cite[Proposition 1.1]{Lau01}) that $D$ can be chosen to be symmetric with respect to $g_0$, so that one always has
\[
\mL_Xg_0=g_0(D)=g_0(D \cdot, \cdot),
\]
for some $D=D_X \in \text{Der}(\mathfrak{g}) \cap \text{Sym}(\mathfrak{g},g_0)$.
\ref{genriccisol} then becomes
\begin{equation} \label{genriccisol_lie}
\text{Rc}_{g_0} = \lambda g_0 + g_0(D) + \frac{1}{4} H_0 \upperset{g_0}{\circ} H_0 + \frac{1}{2}d_{{}^{g_0}}^*H_0 - \frac{1}{2} {\nabla_{g_0,H_0}^+} \theta_0 + \frac{1}{2} \omega \quad \in S^2 \mathfrak{g}^*,
\end{equation}
for $g_0 \in S^2_+\mathfrak{g}^*$, $\lambda \in \mathbb{R}$, $D \in \text{Der}(\mathfrak{g}) \cap \text{Sym}(\mathfrak{g},g_0)$, $H_0 \in \Lambda^3 \mathfrak{g}^*$ (with $d_\mu H_0=0$, $d_\mu \colon \Lambda^3 \mathfrak{g}^* \to \Lambda^4 \mathfrak{g}^*$ denoting the Chevalley-Eilenberg differential of the Lie algebra $(\mathfrak{g},\mu)$), $\theta_0 \in \mathfrak{g}^*$, $\omega \in \Lambda^2\mathfrak{g}^*$, or equivalently
\begin{equation} \label{genriccisol_lie_split}
\begin{dcases} 
\text{Rc}_{g_0} = \lambda g_0 + g_0(D) + \dfrac{1}{4} H_0 \upperset{g_0}{\circ} H_0 - \dfrac{1}{4} \mL_{g_0^{-1}\theta_0} g_0, \\
\omega= -d_{{}^{g_0}}^*H_0 + \frac{1}{2} d\theta_0 - \frac{1}{2} \iota_{g_0^{-1}\theta_0}H_0.
\end{dcases}
\end{equation}

Notice that $d_{{}^{g_0}}^*H_0$ is still a left-invariant form, since the Hodge star operator commutes with pull-backs via orientation-preserving isometries of $g_0$, such as left translations $L_g$, $g \in G$, by left-invariance of $g_0$.

\section{Generalized bracket flows} \label{sec:bracket}

Bracket flows have proven to be a powerful tool in the study of geometric flows on homogeneous spaces. 
This technique was first fully formalized by J. Lauret to study the Ricci flow on nilpotent Lie groups \cite{Lau11}. 
In particular, J. Lauret proved that the Ricci flow on an $n$-dimensional simply connected nilpotent Lie group $G$ starting from a left-invariant Riemannian metric $g_0$ is equivalent to an \textsc{ode} system defined on the variety of nilpotent Lie algebras $\mathcal{N}_n$,
\begin{equation} \label{Lauret_brflow}
	\begin{dcases}
		\dot{\mu}(t)=-\pi(\text{Ric}_{\mu(t)})\mu(t),\\
		\mu(0)=\mu_0,
	\end{dcases}
\end{equation}
where $\mu_0$ is the nilpotent Lie bracket associated with a fixed $g_0$-orthonormal left-invariant frame and $\pi \colon \mathfrak{gl}_n \to \mathfrak{gl}(\mathcal{V}_n)$, given by
\[
(\pi(\phi)\mu)(X,Y)=\phi\mu(X,Y)-\mu(\phi X,Y)-\mu(X, \phi Y), \quad \phi \in \mathfrak{gl}_n,\quad \mu \in \mathcal{V}_n,\quad X,Y \in \R^n,
\]
is the differential of the standard $\text{GL}_n$-action on $\mathcal{V}_n$:
\[
(A \cdot \mu) (X,Y)=A\mu(A^{-1}X,A^{-1}Y),\quad A \in \text{GL}_n,\quad \mu \in \mathcal{V}_n,\quad X,Y \in \R^n.
\]

More generally, in literature many other bracket flows have been considered (see for example \cite{Arr13,EFV15,Lau15,LR15,Lau16,Lau17,AL19}): these can be written in the form
\begin{equation} \label{brflow1}
	\begin{dcases}
		\dot{\mu}(t)=-\pi(\phi(\mu(t)))\mu(t),\\
		\mu(0)=\mu_0,
	\end{dcases}
\end{equation}
for some smooth function $\phi \colon \mathcal{V}_n \to \mathfrak{gl}_n$.
\subsection{Left-invariant Dorfman brackets}

Let $E$ be an exact Courant algebroid over a real Lie group $G$. We shall be interested in the case when $G$ is simply connected and nilpotent, so that we know that $G$ is isomorphic to $(\R^n,\cdot_\mu)$, for some Lie bracket $\mu \in \mathcal{N}_n$.

As we have recalled, there exists a unique cohomology class $[H] \in H^3(\R^n)$ such that, for any $H \in [H]$, $E$ is isomorphic to $E_H =T\R^n \oplus T^*\R^n$, endowed with the inner product $\left<\cdot,\cdot\right>$ in \ref{innerstd} and Dorfman bracket $[\cdot,\cdot]_H$ in \ref{bracketH}.

The whole structure descends to a structure on left-invariant sections, viewed as elements of $\R^n \oplus (\R^n)^*$, if and only if the $3$-form $H$ is left-invariant. 
Explicitly, when $X+\xi,Y+\eta \in \R^n \oplus (\R^n)^*$, the Dorfman bracket $[\cdot,\cdot]_H$ reduces to the operator
\begin{equation} \label{muH_form}
	\begin{aligned}
		{\Mu_H}(X+\xi,Y+\eta)=&\mu(X,Y)-\eta \circ \text{ad}_\mu(X) + \xi \circ \text{ad}_\mu(Y) + \iota_Y\iota_XH \\
		=&\mu(X,Y) - \eta \mu(X,\cdot) + \xi \mu(Y,\cdot) + H(X,Y,\cdot).
	\end{aligned}
\end{equation}
We call such a bilinear operator a (nilpotent) \emph{left-invariant Dorfman bracket}.

As one can check directly and also deduce from the axioms of Courant algebroids, a left-invariant Dorfman bracket is totally skew-symmetric, namely
\[
\left< \bm\mu_H(\cdot, \cdot), \cdot \right> \in \Lambda^3(\R^n \oplus (\R^n)^*)^*.
\]
By a little abuse, we can say $\bm\mu_H \in \Lambda^3(\R^n \oplus (\R^n)^*)^*$, by identifying $\Lambda^3(\R^n \oplus (\R^n)^*)^*$ with a subset of $\bm{\mathcal{V}}_n\coloneqq\Lambda^2(\R^n \oplus (\R^n)^*)^* \otimes (\R^n \oplus (\R^n)^*)$.

We shall denote the set of left-invariant Dorfman brackets on $\R^n$ by $\bm{\mathcal{C}}_n$. By definition, it is clear that
\begin{equation} \label{Cn_iso}
	\bm{\mathcal{C}}_n \xleftrightarrow{\text{1:1}} \{ (\mu,H) \in \mL_n \times \Lambda^3(\R^n)^*, \, d_\mu H=0 \}.
\end{equation}
Equivalently, a quick analysis using the axioms of Courant algebroids and the previous remarks shows that $\bm{\mathcal{C}}_n$ can be identified with the algebraic subset of $\bm{\mathcal{V}}_n$ consisting of all brackets $\bm\mu \in \bm{\mathcal{V}}_n$ such that
\begin{itemize}
\item $\bm\mu \in \Lambda^3\left( \rr \right)^*$,
\item $\bm\mu((\R^n)^*,(\R^n)^*)=0$,
\item $\bm\mu$ satisfies the Jacobi identity.
\end{itemize}

Given any $\bm\mu \in \bm{\mathcal{V}}_n$, one can define the \emph{structure constants}  with respect to the standard basis of $\R^n$ as the $(2n)^3=8n^3$ real numbers $\bm{\mu}_{\ul{\ol{i}}\,\ul{\ol{j}}\,\ul{\ol{k}}}$, $i,j,k=1\dots n$, given by
\begin{align*}
	&\bm\mu (e_i,e_j)=\bm\mu_{\ul{i}\ul{j}\ol{k}} e_k + \bm\mu_{\ul{i}\ul{j}\ul{k}} e^k, \qquad \bm\mu(e_i,e^j)=\bm\mu_{\ul{i}\ol{j}\ol{k}} e_k + \bm\mu_{\ul{i}\ol{j}\ul{k}} e^k, \\
	&\bm\mu (e^i,e_j)=\bm\mu_{\ol{i}\ul{j}\ol{k}} e_k + \bm\mu_{\ol{i}\ul{j}\ul{k}} e^k, \qquad \bm\mu(e^i,e^j)=\bm\mu_{\ol{i}\ol{j}\ol{k}} e_k + \bm\mu_{\ol{i}\ol{j}\ul{k}} e^k.
\end{align*}
Taking $\bm\mu_H \in \bm{\mathcal{C}}_n$, the structure constants are skew-symmetric in all three indices and vanish when two or more indices are overlined. The remaining structure constants are determined by $\mu$ and $H$. More precisely,
\[
(\Mu_H)_{\ul{i}\ul{j}\ol{k}} = \mu_{ij}^k, \quad (\Mu_H)_{\ul{i}\ul{j}\ul{k}}=H_{ijk}.
\]
The set of \emph{nilpotent} left-invariant Dorfman brackets on $\R^n$, denoted by $\bm{\mathcal{N}}_n$, is an algebraic subset of $\bm{\mathcal{V}}_n$ contained in $\bm{\mathcal{C}}_n$. It is easy to see that its elements are exactly those Dorfman brackets $\bm\mu_H$ for which $\mu \in \mathcal{N}_n$.

\subsection{Generalized bracket flows}
To introduce classical bracket flows, one uses the differential of the $\text{GL}_n$-action on $\mathcal{V}_n$. In the same spirit, one can consider the natural $\text{GL}(\rr)$ on $\bm{\mathcal{V}}_n$,
\[
(F \cdot \bm\mu) (z_1,z_2)=F\bm\mu(F^{-1}z_1,F^{-1}z_2),\quad F \in \text{GL}(\rr),\quad \bm\mu \in \bm{\mathcal{V}}_n,\quad z_1,z_2 \in \rr,
\]
which induces an action of $\text{GL}_n \cong \textbf{GL}_n \subset \text{SO}(\rr)$ on $\bm{\mathcal{V}_n}$, preserving both $\bm{\mathcal{C}}_n$ and $\bm{\mathcal{N}}_n$.

Now, identifying $\bm\mu \in \bm{\mathcal{C}}_n$ with $(\mu,H) \in \mathcal{V}_n \times \Lambda^3(\R^n)^*$, it is evident that this action distributes as
\[
A \cdot (\mu,H) = (A \cdot \mu, A \cdot H),
\]
where $A \in \text{GL}_n$ and $A \cdot H \coloneqq (A^{-1})^*H$.

We denote the differential of this action again by $\pi \colon \mathfrak{gl}_n \to \mathfrak{gl}(\bm{\mathcal{V}}_n)$: for $\bm\mu \in \bm{\mathcal{V}}_n$, $\phi \in \mathfrak{gl}_n$ one has
\[
\pi(\phi)\bm\mu = \frac{d}{ds}\Big|_{s=0} (e^{s\phi} \cdot \bm\mu) \in T_{\bm\mu} \bm{\mathcal{V}}_n \cong \bm{\mathcal{V}}_n.
\]
Since the curve $s \mapsto e^{s\phi} \cdot \bm\mu$ is contained in the orbit $\text{GL}_n \cdot \bm\mu$, in this interpretation one has
\begin{equation} \label{tang_gbf}
\pi(\phi)\bm\mu \in T_{\bm\mu} (\text{GL}_n \cdot \bm\mu).
\end{equation}

Following the ideas in the work of J. Lauret (see \cite{Lau11}), these remarks suggest the idea of defining a flow, which we shall refer to as \emph{generalized bracket flow}, on the vector space $\bm{\mathcal{V}}_n$, of the form
\begin{equation} \label{GBF1}
	\begin{cases} \dot{\Mu}(t)= -\pi\big(\phi(\bm\mu(t)\big)\bm\mu(t), \\
		\Mu(0)=\Mu_0,
	\end{cases}
\end{equation}
for some smooth function $\phi \colon \bm{\mathcal{V}}_n \to \mathfrak{gl}_n$ and some $\Mu_0 \in \bm{\mathcal{N}}_n$. By \ref{tang_gbf}, a solution $\Mu(t)$ to \ref{GBF1} satisfies $\dot{\Mu}(t) \in T_{\Mu(t)}(\text{GL}_n \cdot \Mu(t)) \subset T_{\Mu(t)}\bm{\mathcal{N}}_n$ for all $t$, so that the curve $\Mu(t)$ is entirely contained in $\bm{\mathcal{N}}_n$. For this reason, the function $\phi$ may also be defined on $\bm{\mathcal{N}}_n$ only.

The system \ref{GBF1} may be rewritten as the \textsc{ode} system on $\mathcal{N}_n \times \Lambda^3(\R^n)^*$
\begin{equation} \label{GBF2}
	\begin{cases} \dot{\mu}(t)=-\pi\big(\phi(\mu(t),H(t))\big)\mu(t),\\
		\dot{H}(t)=-\pi\big(\phi(\mu(t),H(t))\big)H(t), \\
		\mu(0)=\mu_0 \in \mathcal{N}_n, \\
		H(0)=H_0 \in \Lambda^3(\R^n)^*,\quad d_{\mu_0}H_0=0,
	\end{cases}
\end{equation}
where $\pi$ denotes the differential of the $\text{GL}_n$-action on $\mathcal{V}_n$ or $\Lambda^3(\R^n)^*$.

In what follows, we shall omit the time dependencies of the quantities involved.
Fixing the standard basis $\{e_i\}_{i=1}^n$ for $\R^n$, we shall denote by $\phi_i^j$, $i,j=1 \dots n$ the entries of the generic $\phi \in \text{GL}_n$ with respect to it, such that $\phi(e_i)=\phi_i^j e_j$ for all $i=1 \dots n$. One can then compute the coordinate expression for the evolution equations \ref{GBF2}, obtaining
\begin{align}
	\label{GBF_mu} \dot{\mu}_{ij}^k &= \phi_i^l \, \mu_{lj}^k + \phi_j^l \, \mu_{il}^k - \phi_l^k \, \mu_{ij}^l, \\
	\label{GBF_H} \dot{H}_{ijk} &= \phi_i^l \, H_{ljk} + \phi_j^l H_{ilk} + \phi_k^l \, H_{ijl},
\end{align}
for $i,j,k=1,\ldots,n$.

Special generalized bracket flows are obtained when the $\mathfrak{gl}_n$-valued smooth function $\phi$ only depends on $\mu$, $\phi=\phi(\mu)$: when this happens, the first equation of \ref{GBF2} is independent from the second one and corresponds to a usual bracket flow \ref{brflow1} on $\mathcal{N}_n$. 

Classical bracket flows have proved to be a powerful tool in the study of geometric flows on (nilpotent) Lie groups. We thus expect the generalized bracket flows we have defined to be useful in the context of geometric flows in generalized geometry.

\section{Examples on the Heisenberg group} \label{examples_heis}
In this section we perform explicit computations for the constructions introduced in the previous sections. We focus in particular on the Heisenberg group.

The Heisenberg group $H_3$ is a three-dimensional simply connected Lie group, which can be defined as  closed subgroup of $\text{GL}_3$:
\[
H_3=\left\{ \begin{pmatrix} 1&a&c \\ 0&1&b \\ 0&0&1 \end{pmatrix} \in \text{GL}_3,\, a,b,c \in \R \right\}.
\]
Via the exponential map, $H_3$ is diffeomorphic to its Lie algebra
\[
\mathfrak{h}_3=\left\{ \begin{pmatrix} 0 & a & c \\ 0 & 0 & b \\ 0 & 0 & 0 \end{pmatrix} \in \mathfrak{gl}_{3},\, a,b,c \in \R \right\}.
\]
Fixing the basis
\begin{equation} \label{basis}
e_1=\begin{pmatrix} 0&1&0 \\ 0&0&0 \\ 0&0&0 \end{pmatrix},\quad e_2=\begin{pmatrix} 0&0&0 \\ 0&0&1 \\ 0&0&0 \end{pmatrix}, \quad e_3=\begin{pmatrix} 0&0&1 \\ 0&0&0 \\ 0&0&0 \end{pmatrix}
\end{equation}
for $\mathfrak{h}_3$, the induced bracket $\mu \in \mathcal{N}_3$ is $\mu=e^1 \wedge e^2 \otimes e_3$, since $[e_1,e_2]=e_3$, $[e_1,e_3]=[e_2,e_3]=0$.
\subsection{Generalized Ricci solitons on the Heisenberg group}
\label{soliton_Heis}

Let $H_3$ be the Heisenberg group, and fix the basis \ref{basis} for its Lie algebra $\mathfrak{h}_3$.

In order to find generalized Ricci solitons on $H_3$, we first notice that the codifferential $d^*_{g_0}$ is the null map for every $g_0 \in S^2_+\mathfrak{h}_3^*$, since $\upperset{g_0}{*}$ sends $\Lambda^3 \mathfrak{h}_3^*$ to $\R$ and $d \colon \R \to \mathfrak{h}_3^*$ is the null map. With respect to the basis $\{e_1,e_2,e_3\}$ in \ref{basis}, the generic derivation $D$ of $\mathfrak{h}_3$ can be written in matrix form as
\[
D=\begin{pmatrix} a_1 & a_2 & 0 \\
                  a_3 & a_4 & 0 \\
                  a_5 & a_6 & a_1+a_4
  \end{pmatrix},
\]
with $a_i \in \R$, $i=1 \dots 6$.

Let $g_0$ be the standard metric
\[
g_0=e^1 \otimes e^1+e^2 \otimes e^2+e^3 \otimes e^3,
\]
such that $\{e_1,e_2,e_3\}$ is an orthonormal basis. Now, symmetric derivations with respect to $g_0$ are simply represented by symmetric matrices with respect to this basis:
\begin{equation} \label{deriv}
D=\begin{pmatrix} a_1 & a_2 & 0 \\ a_2 & a_3 & 0 \\ 0 & 0 & a_1+a_3 \end{pmatrix},
\end{equation}
$a_i \in \R$, $i=1,2,3$.
In what follows, assume
\[
H_0=a \, e^{123} = \frac{a}{6} \, \varepsilon_{ijk} \, e^{ijk}, \qquad
\theta_0= \theta_i \, e^i, \qquad
\omega= \frac{1}{2} \, \omega_{ij} \, e^{ij},
\]
where $a, \theta_i, \omega_{ij} \in \R$, $\omega_{ij}=-\omega_{ji}$, $i,j=1,2,3$, $e^{i_1 \dots i_k} \coloneqq e^{i_1} \wedge \dots \wedge e^{i_k}$ and $\varepsilon_{ijk}$ is equal to the sign of the permutation sending $(1,2,3)$ into $(i,j,k)$ whenever $i$, $j$ and $k$ are all different, and equal to $0$ otherwise, by definition.

We are now ready to compute the coordinate expression for all the terms involved in \ref{genriccisol_lie}:
\begin{itemize}
\item $\text{Rc}_{g_0}$: from \ref{riccinil_coord}, since the basis $\{e_1,e_2,e_3\}$ is orthonormal, by a direct computation we get
\[
\text{Rc}_{g_0}=\begin{pmatrix} -\frac{1}{2} & 0 & 0 \\ 0 & -\frac{1}{2} & 0 \\ 0 & 0 & \frac{1}{2} \end{pmatrix},
\]
in the fixed basis,
\item $g_0(D)$: it is simply represented by the matrix \ref{deriv} with respect to the orthonormal basis,
\item $H_0 \upperset{g_0}{\circ} H_0$: one has $H_0 \upperset{g_0}{\circ} H_0(e_i,e_j)= g(\iota_{e_i} H_0, \iota_{e_j} H_0)=g_0^{rl}g_0^{st}(H_0)_{irs}(H_0)_{jlt}=a^2 \, \varepsilon_{ist}\, \varepsilon_{jst}$, so that, in matrix form, we get
\[
H_0 \upperset{g_0}{\circ} H_0=\begin{pmatrix} 2a^2 & 0 & 0 \\ 0 & 2a^2 & 0 \\ 0 & 0 & 2a^2 \end{pmatrix},
\]
\item ${\nabla_{{g_0},H_0}^+} \theta_0$: writing $\nabla^+$ instead of $\nabla_{g_0,H_0}^+$ and by left-invariance of the quantities involved, one has
\[
{\nabla^+} \theta_0(e_i,e_j)=-\theta_0(\nabla^+_{e_i}{e_j}).
\]
Now, $\nabla^+=\nabla^{g_0} + \frac{1}{2} g_0^{-1}H_0$ and, letting $\nabla^{g_0}_{e_i}e_j=\Gamma_{ij}^k e_k$ and recalling the Koszul formula one computes
\[
\Gamma_{ij}^k=-\frac{1}{2} \left( {\mu_{jk}^i}+{\mu_{ik}^j} + {\mu_{ji}^k} \right),\quad \frac{1}{2}g_0^{-1}H_0(e_i,e_j)=\frac{1}{2}a \, \varepsilon_{ijk} \, e_k,
\]
so that
\[
{\nabla^+} \theta_0(e_i,e_j)=\frac{1}{2} \theta_k \left({\mu_{jk}^i}+{\mu_{ik}^j} + {\mu_{ji}^k} - a \, \varepsilon_{ijk} \right).
\]
The corresponding matrix with respect to the orthonormal basis is thus
\[
{\nabla^+} \theta_0=\begin{pmatrix}
0 & -\frac{1}{2}\theta_3(1+a) & \frac{1}{2} \theta_2 \left( 1 + a \right) \\
\frac{1}{2} \theta_3 (1+a) & 0 & -\frac{1}{2}\theta_1 \left( 1 +a \right) \\
\frac{1}{2} \theta_2 \left(1-a \right) & -\frac{1}{2} \theta_1 \left(1-a \right) & 0
\end{pmatrix}.
\]
so that its symmetric and skew-symmetric parts are
\begin{align*}
S({\nabla^+} \theta_0)=& \begin{pmatrix} 0 & 0 & \frac{1}{2} \theta_2 \\
0 & 0 & -\frac{1}{2} \theta_1 \\ \frac{1}{2} \theta_2 & -\frac{1}{2} \theta_1 & 0 \end{pmatrix}, 
\\
A({\nabla^+} \theta_0)=& \begin{pmatrix}
0 & -\frac{1}{2} \theta_3(1+a) & \frac{1}{2} a \,\theta_2 \\
\frac{1}{2} \theta_3(1+a) & 0 & -\frac{1}{2} a \, \theta_1 \\
-\frac{1}{2} a\, \theta_2 & \frac{1}{2} a\, \theta_1 & 0
\end{pmatrix}.
\end{align*}
\end{itemize}

The first equation of \ref{genriccisol_lie_split} gives now rise to a system of six equations in the unknowns $\lambda, a_1 , a_2 , a_3,\theta_1 , \theta_2 , \theta_3$:
\[
\begin{dcases*}
\frac{1}{2} + \lambda + a_1 + \frac{1}{2}a^2=0,\\
a_2=0,\\
\theta_2=0,\\
\frac{1}{2} + \lambda + a_3 + \frac{1}{2}a^2=0,\\
\theta_1=0,\\
-\frac{1}{2} + \lambda + a_1+a_3+\frac{1}{2} a^2=0,
\end{dcases*}
\]
which is equivalent to
\[
\begin{dcases*}
\lambda=-\frac{1}{2}(3+a^2),\\
a_1=a_3=1,\\
a_2=\theta_1=\theta_2=0.
\end{dcases*}
\]
The second equation of \ref{genriccisol_lie_split} now implies
\[
\omega_{12}=-\omega_{21}=-\tfrac{1}{2}\theta_3(1+a),
\]
while all the other $\omega_{ij}$'s vanish.

We thus obtain generalized Ricci solitons with the data
\[
\begin{dcases*}
g_0=e^1 \otimes e^1 + e^2 \otimes e^2 + e^3 \otimes e^3,\\
\lambda= -\frac{1}{2}(3+a^2),\\
D=e^1 \otimes e_1 + e^2 \otimes e_2 + 2 \, e^3 \otimes e_3,\\
H_0= a\, e^{123},\\
\theta_0 = \theta_3\,e^3,\\
\omega= -\tfrac{1}{2}\theta_3(1+a) e^{12}.
\end{dcases*}
\]

\begin{remark}
The metric $g_0$ above is actually also a Ricci soliton in the classical sense, since, setting $a=\theta_3=0$, $H_0$, $\theta_0$ and $\omega$ vanish, leaving $g_0$ satisfying $\text{Rc}_{g_0}= \lambda g_0 + g_0(D)$, or equivalently, applying $g_0^{-1}$, $\text{Ric}_{g_0}= \lambda \operatorname{Id} + D$ for $\lambda=-\frac{3}{2}$ and $D$ as above. By \cite[Theorem 3.5]{Lau01}, $g_0$ is the only left-invariant Ricci soliton on $H_3$, up to isometry and rescaling.
\end{remark}

\subsection{A generalized bracket flow on the Heisenberg group}
The definition of the gauge-corrected generalized Ricci flow \ref{GRF_gauge} suggest the generalized bracket flow
\begin{equation} \label{genbrflow_ricH}
\begin{dcases}
	\dot{\bm\mu}(t)=-\pi\left(\text{Ric}_{\mu(t)} - \tfrac{1}{4} H(t)^2 \right)\bm\mu(t), \\
	\bm\mu(0)=\bm\mu_0 \in \bm{\mathcal{N}}_n.
\end{dcases}
\end{equation}
Here, for every $H \in \Lambda^3(\R^n)^*$, we denote $H^2 \coloneqq \left<\cdot,\cdot\right>^{-1} (H \circ H)$, where $\circ$ is meant with respect to $\left< \cdot, \cdot \right>$. Recalling \ref{riccinil_coord}, the whole endomorphism $\phi(\bm\mu)=\phi(\mu,H)=\text{Ric}_{\mu} - \tfrac{1}{4} H^2$ can then be written in coordinates as
\[
\phi_i^j=-\frac{1}{2} \mu_{ik}^l \mu_{jk}^l + \frac{1}{4} \mu_{kl}^i \mu_{kl}^j - \frac{1}{4} H_{ikl} H_{jkl},
\]
with respect to the standard basis of $\R^n$

Now, let $n=3$ and let $\mu_0$ be the Heisenberg Lie bracket $\mu_0=e^{1} \wedge e^{2} \otimes e_3$. Let $H_0$ be the generic (trivially $d_{\mu_0}$-closed) $3$-form $H_0=c\,e^{123}$, $c \in \R$. Then, using \ref{riccinil_coord}, \ref{GBF_mu} and \ref{GBF_H}, it is easy to see to compute that the solution to \ref{genbrflow_ricH} is of the form $\bm\mu(t)=\left(x(t)\mu_0,y(t)e^{123}\right)$, with $x(t)$ and $y(t)$ satisfying the \textsc{ode} system
\begin{equation} \label{genbrflow_ex}
\begin{cases}
	\dot{x}=-\frac{3}{2} x^3 - \frac{1}{2} xy^2, \\
	\dot{y}=-\frac{3}{2} y^3 - \frac{1}{2} x^2y, \\
	x(0)=1, \, y(0)=a.
\end{cases}
\end{equation}
It is easy to see that the solution $(x(t),y(t))$ is defined for all positive times and converges to $(0,0)$, since
\[
\frac{d}{dt} \left(x(t)^2 + y(t)^2\right)= 2 (x(t)\dot{x}(t) + y(t)\dot{y}(t)) \leq -\left(x(t)^2+y(t)^2\right)^2,
\]
so that, by comparison, we get
\[
x(t)^2 + y(t)^2 \leq \frac{1+a^2}{1+(1+a^2)t},
\]
for all $t \geq 0$.
For $a=1$, the explicit solution to \ref{genbrflow_ex} is given by
\[
x(t)=y(t)=(1+4t)^{-\frac{1}{2}}.
\]
defined for $t \in (-\tfrac{1}{4},\infty)$.

\section{Generalized Ricci flow on the Heisenberg group} \label{genricciflow_heis}
Let us consider the gauge-corrected generalized Ricci flow \ref{GRF_gauge} on the three-dimensional Heisenberg group $H_3$, with initial data
\[
g_0=e^1 \otimes e^1 + e^2 \otimes e^2 + e^3 \otimes e^3,\quad H_0=a\,e^{123},\,a \in \R.
\]
Denoting by $(g(t),H(t))$ the solution at time $t$, we adopt the ansatz
\[
g(t)=g_1(t)\,e^1 \otimes e^1 + g_2(t)\,e^2 \otimes e^2 + g_3(t)\,e^3 \otimes e^3,\quad g_i(0)=1,\,i=1,2,3,
\]
while $H(t)=H_0$ is necessarily constant since $d^*_{g}$ and $\Delta_{g}$ are null maps for every left-invariant Riemannian metric $g$, as remarked in Subsection \ref{soliton_Heis}.

An explicit computation yields
\[
	\text{Rc}_{g(t)}=\begin{pmatrix} -\frac{1}{2} \frac{g_3}{g_2} &0&0 \\ 0& -\frac{1}{2} \frac{g_3}{g_1} &0 \\ 0&0& \frac{1}{2} \frac{g_3^2}{g_1g_2} \end{pmatrix},\quad H_0 \circ H_0=\begin{pmatrix} \frac{2a^2}{g_2g_3} &0&0 \\ 0&  \frac{2a^2}{g_1g_3} &0 \\ 0&0& \frac{2a^2}{g_1g_2} \end{pmatrix},
\]
so that \ref{GRF_gauge} reduces to the \textsc{ode} system
\[
\begin{dcases*}
\dot{g}_1=\frac{g_3}{g_2} + \frac{a^2}{g_2g_3},\\
\dot{g}_2=\frac{g_3}{g_1} + \frac{a^2}{g_1g_3},\\
\dot{g}_3=-\frac{g_3^2}{g_1g_2} + \frac{a^2}{g_1g_2},\\
g_i(0)=1,\quad i=1,2,3.
\end{dcases*}
\]
By uniqueness, we thus have $g_1(t)=g_2(t)$ for all $t$ and we obtain
\begin{equation} \label{GRF_H}
\begin{dcases}
	\dot{g}_1=\frac{a^2+g_3^2}{g_1g_3},\\
	\dot{g}_3=\frac{a^2-g_3^2}{g_1^2},\\
	g_1(0)=g_3(0)=1.
\end{dcases}
\end{equation}
Special cases are given by
\begin{itemize}
\item $a=0$: the generalized Ricci flow reduces to the classical Ricci flow and an explicit solution to \ref{GRF_H} is given by
\[
g_1(t)=(1+3t)^{\frac{1}{3}},\quad g_3(t)=(1+3t)^{-\frac{1}{3}},
\]
defined on the maximal definition interval $I=(-\frac{1}{3},\infty)$ (cf. \cite{IJ92}),
\item $a= \pm 1$: the system reduces to 
\[
\begin{dcases*}
\dot{g}_1=\frac{2}{g_1},\\
\dot{g}_3=0,	\\
g_1(0)=g_3(0)=1,
\end{dcases*}	
\]
with solution
\[
g_1(t)=(1+4t)^{\frac{1}{2}},\quad g_3(t)=1,
\]
for $t \in I=(-\frac{1}{4},\infty)$.
\end{itemize}
A quick qualitative analysis of \ref{GRF_H} shows that, for all $a \in \R$, the solution to \ref{GRF_H} exists for all positive times, with
\[
\lim_{t \to \infty} g_1(t)= \infty, \quad \lim_{t \to \infty} g_3(t)= \lvert a \rvert.
\]
The maximal definition interval is always of the form $I_a=(T_{\text{min}}(a),\infty)$, where $T_{\text{min}} \colon \R \to \R_{<0}$ is an even function, with $T_{\text{min}}(0)=-\frac{1}{3}$ and monotonically converging to $0$ as $a$ goes to infinity (see Figure \ref{Tmin}). We also have
\[
\lim_{t \to T_{\text{min}}(a)^+} g_1(t)=0,\quad \lim_{t \to T_{\text{min}}(a)^+} g_3(t)= 
\begin{cases}
	\infty&\quad \lvert a \rvert < 1,\\
	1&\quad a=\pm 1, \\
	0&\quad \lvert a \rvert > 1.
\end{cases}	
\]

In Figure \ref{fig_sol}, we show some solutions of \ref{GRF_H}, sampled for $a=\frac{k}{4}$, $k=0,\ldots,9$, and viewed as curves in the phase plane $(g_1,g_3)$. The red and blue curves correspond to $a=0$ and $a=1$, respectively.

\begin{figure}
\centering
\includegraphics[width=0.8\textwidth]{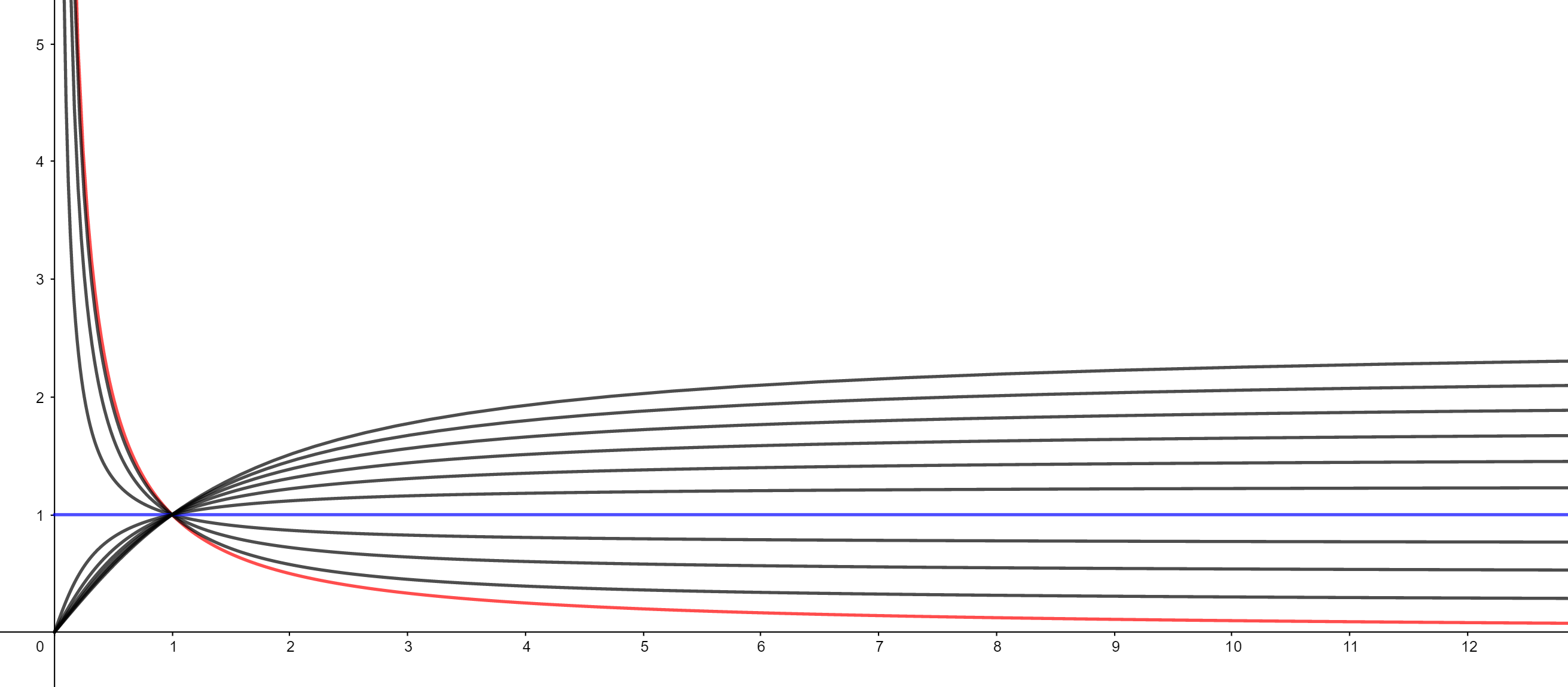}
\caption{Examples of solutions to \ref{GRF_H} viewed in the phase plane $(g_1,g_3)$.}
\label{fig_sol}
\end{figure}


\begin{figure}
	\centering
	\includegraphics[width=0.8\textwidth]{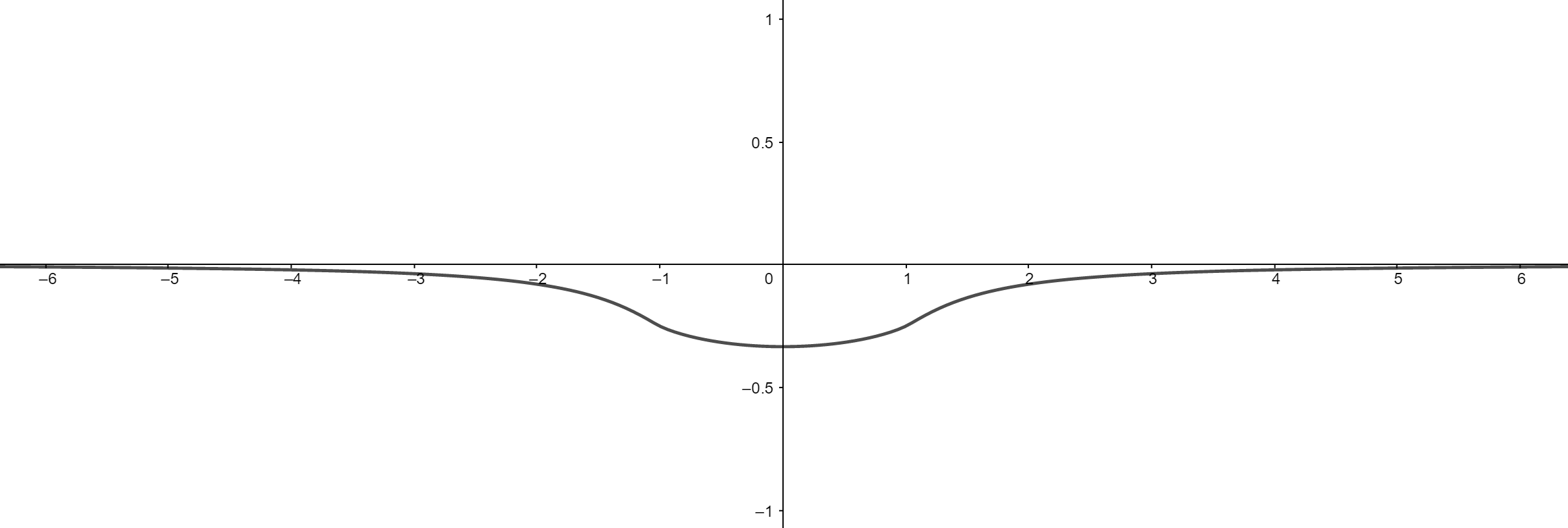}
	\caption{Behavior of the map $T_{\text{min}}$.}
	\label{Tmin}
\end{figure}
 
\medskip
\begin{footnotesize}

\medskip
\end{footnotesize}
\end{document}